\documentclass{article}
\usepackage[utf8]{inputenc}
\usepackage[T2A]{fontenc}
\usepackage[english]{babel}

\usepackage{parskip}

\usepackage{amsmath}
\usepackage{amssymb}
\usepackage{amsthm}
\usepackage{amsfonts}
\usepackage{asymptote}
\usepackage{xcolor}
\usepackage{subcaption}
\usepackage{graphicx}
\usepackage{cancel}

\newtheorem{lemma}{Lemma}[]

\newtheorem{proposition}{Proposition}[]

\newtheorem{theorem}{Theorem}[]
\newtheorem{corollary}{Corollary}[]
\theoremstyle{definition}
\newtheorem{definition}{Definition}[]

\theoremstyle{plain}

\newcommand{\diam}{\operatorname{diam}}

\newcommand{\dis}{\operatorname{dis}}
\newcommand{\dist}{\operatorname{d}}

\newcommand{\dil}{\operatorname{dil}}

\renewcommand{\:}{\colon}

\renewcommand{\ss}{\subset}

\newcommand{\N}{\mathbb{N}}
\newcommand{\R}{\mathbb{R}}
\newcommand{\Z}{\mathbb{Z}}

\title{Lipschitz distance between clouds}
\author{Ivan N. Mikhailov}
\date{}

\begin{document}
\maketitle

\begin{abstract}
In this note we show that the Lipschitz distance between the classes of metric spaces at finite Gromov--Hausdorff distances from the one-point metric space~$\Delta_1$ and the real line~$\R$ with the natural metric, respectively, is positive. 
\end{abstract}



\section{Introduction}

In~\cite{GromovEng} M.\,Gromov introduced moduli spaces of the class of all metric spaces at finite Gromov--Hausdorff distances from a given metric space. It was mentioned that such moduli spaces are always complete and contractible (\cite{GromovEng}[section $3.11_{\frac{1}{2}_+}$]). In~\cite{BogatyTuzhilin} the authors suggested to work with such moduli spaces (they were called \emph{clouds}) in the sense of NBG set theory to avoid arising set-theoretic issues. While the completeness of each cloud was verified in~\cite{BogatyTuzhilin}, the contractibility of each cloud remains an open question for a number of reasons. The main issue here is that a natural homothety-mapping that takes a metric space $(X, d_X)$ into $(X, \lambda d_X)$ for some $\lambda > 0$ and generates a contraction of a cloud of all bounded metric spaces if $\lambda\to 0$, does not behave so well in case of unbounded metric spaces. Firstly, in~\cite{BogatyTuzhilin} it was shown that there exist metric spaces such that $d_{GH}(X, \lambda X) = \infty$ for some $\lambda > 0$. The simplest one is a geometric progression $X = \{3^n\:n\in\N\}$ with a natural metric, for which $d_{GH}(X, 2X) = \infty$. Secondly, even for clouds that are invariant under multiplication on all positive numbers a homothety-mapping may not be continuous. In~\cite{nglitghclitcotrl} it was shown that $d_{GH}(\Z^n, \lambda \Z^n)\ge \frac{1}{2}$ for all $\lambda > 1$, $n\in\N$.

In~\cite{Nesterov} the global geometry of clouds was studied. Namely, each cloud equipped with the Gromov--Hausdorff distance is a metric class, and, hence, the Gromov--Hausdorff distance between different clouds can be defined. It was shown~\cite{Nesterov}[Corollary $3$] that the Gromov--Hausdorff distance between the cloud of bounded metric spaces and the cloud of the real line is infinite. 

In this note we consider the Lipschitz distance between clouds instead of the Gromov--Hausdorff distance. By the idea of Alexey Tuzhilin, since the homothety-mapping is continuous with respect to the Lipschitz distance, this trick could allow to develop a sound geometry of clouds. However, a tricky problem appears immediately: is there a pair of clouds on a finite and non-zero Lipschitz distance from each other? By applying the ideas from~\cite{Nesterov}, we show that the Lipschitz distance between the cloud of bounded metric spaces and the cloud of the real line is positive. However, the question whether this distance is finite remains open.

\subsection*{Acknowledgements}

The author is grateful to his scientific advisor A.\,A.\,Tuzhilin and A.\,O.\,Ivanov for the constant attention to the work. 

Also the author is grateful to all the participants of the seminar <<Theory of extremal networks>>, leaded by A.\,O.\,Ivanov and A.\,A.\,Tuzhilin in Lomonosov Moscow State University for numerous fruitful discussions.

\section{Preliminaries}

\emph{A metric space} is an arbitrary pair $(X,\,d_X)$, where $X$ is an arbitrary set, $d_X\: X\times X\to [0,\,\infty)$ is some metric on it, that is, a nonnegative symmetric, positively definite function that satisfies the triangle inequality. 

For convenience, if it is clear in which metric space we are working, we denote the distance between points $x$ and~$y$ by~$|xy|$. Suppose $X$ is a metric space. By $U_r(a) =\{x\in X\colon |ax|<r\}$, $B_r(a) = \{x\in X\colon |ax|\le r\}$ we denote open and closed balls centered at the point~$a$ of the radius~$r$ in~$X$. For an arbitrary subset $A\subset X$ of a metric space~$X$, let $U_r(A) = \cup_{a\in A} U_r(a)$ be the open $r$-neighborhood of~$A$. For non-empty subsets $A \ss X$, $B \ss X$ we put $\dist(A,\,B)=\inf\bigl\{|ab|:\,a\in A,\,b\in B\bigl\}$.

\subsection{Hausdorff and Gromov--Hausdorff distances}

\begin{definition}
Let $A$ and $B$ be non-empty subsets of a metric space.
\emph{The Hausdorff distance} between $A$ and~$B$ is the quantity $$d_H(A,\,B) = \inf\bigl\{r > 0\colon A\subset U_r(B),\,B\subset U_r(A)\bigr\}.$$
\end{definition}
\begin{definition} Let $X$ and $Y$ be metric spaces. The triple $(X', Y', Z)$, consisting of a metric space $Z$ and its two subsets $X'$ and $Y'$, isometric to $X$ and~$Y$, respectively, is called \emph{a realization of the pair} $(X, Y)$.
\end{definition}

\begin{definition} \emph{The Gromov-Hausdorff distance} $d_{GH} (X, Y)$ between $X$ and~$Y$ is the exact lower bound of the numbers $r\ge 0$ for which there exists a realization $(X', Y', Z)$ of the pair $(X, Y)$ such that $d_H(X',\,Y') \le r$. 
\end{definition}

Now let $X,\,Y$ be non-empty sets.  

\begin{definition} Each $\sigma\subset X\times Y$ is called a \textit{relation} between $X$ and~$Y$.
\end{definition}

By $\mathcal{P}_0(X,\,Y)$ we denote the set of all non-empty relations between $X$ and~$Y$.

We put $$\pi_X\colon X\times Y\rightarrow X,\;\pi_X(x,\,y) = x,$$ $$\pi_Y\colon X\times Y\rightarrow Y,\;\pi_Y(x,\,y) = y.$$ 

\begin{definition} A relation $R\subset X\times Y$ is called a \textit{correspondence}, if their restrictions $\pi_X|_R$ and $\pi_Y|_R$ are surjective.
\end{definition}

Let $\mathcal{R}(X,\,Y)$ be the set of all correspondences between $X$ and~$Y$.

\begin{definition} Let $X,\,Y$ be metric spaces, $\sigma \in \mathcal{P}_0(X,\,Y)$. The \textit{distortion} of $\sigma$ is the quantity $$\dis \sigma = \sup\Bigl\{\bigl||xx'|-|yy'|\bigr|\colon(x,\,y),\,(x',\,y')\in\sigma\Bigr\}.$$
\end{definition}

\begin{proposition}[\cite{BBI}]  \label{proposition: distGHformula}
For arbitrary metric spaces $X$ and~$Y$, the following equality holds $$2d_{GH}(X,\,Y) = \inf\bigl\{\dis\,R\colon R\in\mathcal{R}(X,\,Y)\bigr\}.$$
\end{proposition}

\subsection{Lipschitz distance}

\begin{definition} \label{def: bilip1}
A homemorphism~$f$ of metric spaces is called \emph{bilipschitz} iff both maps~$f$ and~$f^{-1}$ are Lipschitz. 
\end{definition}

\begin{definition}\label{def: bilip2} \emph{The Lipschitz distance} between metric spaces $X$ and~$Y$ is the value  
$$d_L(X,\,Y) = \inf_{f\colon X\to Y}\,\ln\bigl(\max\bigl\{\dil(f),\,\dil(f^{-1})\bigr\}\bigr),$$
where infinum is taken over all bilipschitz homeomorphisms~$f\colon X\to Y$.
\end{definition}

\subsection{Clouds}

By~$\mathcal{VGH}$ we denote the class of all non-empty metric spaces, equipped with the Gromov--Hausdorff distance.  

Note that~$\mathcal{VGH}$ is a proper class in the sense of the NBG set theory. In this theory all objects are \emph{classes} of one of the two following types: \emph{sets}, or \emph{proper classes}. A class is called a set if it belongs to some other class, and a proper class otherwise. It is important for us that for all classes the following natural constructions are defined: Cartesian product, maps between classes, metrics, pseudometrics, etc.  

\begin{theorem}[\cite{BBI}]
The Gromov--Hausdorff distance is a generalised pseudometric on~$\mathcal{VGH}$, vanishing on pairs of isometric metric spaces. Namely, the Gromov--Hausdorff distance is symmetric, satisfies a triangle inequality, though can vanish or be infinite between some pairs of non-isometric metric spaces.
\end{theorem}

A class~$\mathcal{GH}_0$ is obtained from~$\mathcal{VGH}$ by factorization over zero distances, i.e., over an equivalence relation $X\sim_0 Y$, iff $d_{GH}(X,\,Y) = 0$.

\begin{definition}
Consider an equivalence relation~$\sim_1$ on~$\mathcal{GH}_0$: $X\sim_1 Y$, iff~$d_{GH}(X,\,Y) < \infty$. We call the corresponding equivalence classes \emph{clouds}.
\end{definition}

For an arbitrary metric space~$X$, we denote a cloud containing~$X$ by~$[X]$. Let~$\Delta_1$ be a metric space, consisting of a single point. Hence, $[\Delta_1]$ is the cloud of all bounded metric spaces. 

Suppose that for some metric spaces $A$ and~$A'$, the equality holds~$d_{GH}(A, A') = 0$. Then, for arbitrary metric space~$B$, we also have~$d_{GH}(A, B) = d_{GH}(A', B)$. From this simple observation, it follows that all the results about~$d_{GH}(A, B)$ also hold if we exchange $A$ by~$A'$ such that~$d_{GH}(A, A') = 0$. Thus, instead of interpreting <<$A\in [X]$>> directly by definition so that $A$ is an equivalence class of all metric spaces on zero Gromov--Hausdorff distances from each other, we will mean that~$A$ is a certain member of this equivalence class. For example,~$X\in[\Delta_1]$ can be read as <<$X$ is a bounded metric space>> throughout the paper.

\begin{theorem}[\cite{BBI}] \label{thm: boundedcloud}
Let $X$ and~$Y$ be arbitrary bounded metric spaces. Then
\begin{itemize}
\item The inequalities hold 
$$\frac{1}{2}\bigl|\diam X - \diam Y\bigr|\le d_{GH}(X, Y)\le \max\bigl\{d_{GH}(X, \Delta_1), d_{GH}(Y, \Delta_1)\bigr\} = \frac{1}{2}\max\bigl\{\diam X, \diam Y\bigr\}.$$
\item A map $\Phi\:[\Delta_1]\times \R_{\ge 0}\to [\Delta_1]$, $\Phi(X, \lambda) = \lambda X$ is continuous and generates a contraction of the cloud~$[\Delta_1]$ if~$\lambda \to 0$.
\item A curve~$\lambda X$, $\lambda\in[0, +\infty)$ is a geodesic with respect to the Gromov--Hausdorff distance in the cloud~$[\Delta_1]$.
\end{itemize}   
\end{theorem}

\begin{theorem}[\cite{Nesterov}] 
Every cloud is a proper class in the sense of NBG set theory. 
\end{theorem}

Since in NBG theory each proper class can be put into one-to-one correspondence with the class of all ordinals \cite{properclassesbijection}[p. 53], we obtain

\begin{corollary}\label{cor: cloud-bijection}
Every two clouds are bijective to each other.
\end{corollary}

From Corollary~\ref{cor: cloud-bijection}, it follows that, if we equip each cloud with the Gromov--Hausdorff distance, we can define the Lipschitz distance between them by mimicking Definitions~\ref{def: bilip1}, and~\ref{def: bilip2}. To avoid unnecessary technicalities, in these definitions by of $f$ and $f^{-1}$ we mean the usual continuity with respect to the metric defined Gromov--Hausdorff distance (a map $f\: A\to  B$ is continuous in $x$ if, for arbitrary $\varepsilon > 0$, there exists $\delta > 0$ such that $f\bigl(U_\delta(x)\bigr)\subseteq U_\varepsilon(f(x))$).

Finally, we will need the following construction.

Let $X$ be arbitrary bounded path-connected metric space of diameter~$1$. Fix $0 < \delta < \frac{1}{2}$. We put
\begin{align*}
&\Z_t = \cup_{n\in\Z}[n-t, n+t]\subset \R, \;t\in \Bigl[\frac{1}{2}-\delta, \frac{1}{2}\Bigr],\\
& \R_d = \R\times_{\ell^1}(dX),\;d\in [0, \delta],
\end{align*}
where $X\times_{\ell^1}Y$ is the Cartesian product $X\times Y$ equipped with the $\ell^1$-metric:
$$d_{X\times_{\ell^1}Y}\bigl((x, y), (x', y')\bigr) = d_X(x, x') + d_Y(y, y').$$

\begin{theorem}[\cite{geodesicthroughR}]\label{theorem: geodesicthroughR}
By gluing $($and reparametrizing$)$ $\Z_t$, $t\in [\frac{1}{2}-\delta, \frac{1}{2}]$ and $\R_d$, $d\in [0, \delta]$, we obtain a shortest curve in the Gromov--Hausdorff class, for which $\R$ is an interior point.  
\end{theorem}

\section{Main theorem}

\begin{lemma}\label{LipDefinition}
Let $X$ and $Y$ be arbitrary metric classes. Then  $d_L(X,\,Y) = 0$ iff for arbitrary $1 > \varepsilon > 0$ there exists a bijeciton  $f\colon X\to Y$ such that for all $x,\,x'\in X$ the following inequalities hold 
\begin{align}\label{LipIneq}
(1-\varepsilon)|xx'|&\le |f(x)f(x')|\le (1+\varepsilon)|xx'|,\\ 
(1-\varepsilon)|f(x)f(x')|&\le |xx'|\le (1+\varepsilon)|f(x)f(x')|.
\end{align}
\end{lemma}

\begin{proof}
Suppose $d_L(X, Y) = 0$. Then for arbitrary $\delta > 0$ there exists a bilipschitz homeomorphism $f\:X\to Y$ such that $1\le \max\bigl\{\dil f, \dil f^{-1}\bigr\}\le e^\delta$. Hence, by Definition~\ref{def: bilip1}, for all $x, x'\in X$, the inequalities hold
\begin{align}\label{auxillary-inequialities}
    |xx'|\le e^\delta|f(x)f(x')|,\; |f(x)f(x')|\le e^\delta |xx'|.
\end{align}
Fix some $1 > \varepsilon > 0$. Then by initially choosing $\delta>0$ such that $\displaystyle 1-\varepsilon\le e^{-\delta},\, e^\delta \le 1+\varepsilon$ which is equivalent to $\delta \le\min\bigl\{\ln(1+\varepsilon),\,-\ln(1-\varepsilon)\bigr\}$, the desired inequalities follow from Inequialities~\ref{auxillary-inequialities}.

Conversely, to ensure that $d_L(X, Y) = 0$ follows from the stated inequalities, we have to choose $\varepsilon > 0$ such that $\displaystyle e^{-\delta} \le 1-\varepsilon,\, 1+\varepsilon\le e^\delta$ which is equivalent to $\varepsilon \le \min\bigl\{1-e^{-\delta},\,e^\delta-1\bigr\}$. 

\end{proof}

\begin{theorem}\label{theorem: nonzeroLip}
The inequality holds $d_L\bigl([\mathbb{R}],\,[\Delta_1]\bigr) > 0$. 
\end{theorem}

\begin{proof}
Suppose that $d_L\bigl([\Delta_1],\,[\mathbb{R}]\bigr) = 0$. By Lemma~\ref{LipDefinition}, it is equivalent to the following condition: for arbitrary~$\varepsilon > 0$, there exists a bijection~$f\: [\Delta_1]\to [\R]$ such that for all~$x,\,x'\in [\Delta_1]$ the inequalities hold 
\begin{align}\label{LipIneq2}
(1-\varepsilon)d_{GH}(x, x')\le d_{GH}\bigl(f(x),f(x')\bigr)\le (1+\varepsilon)d_{GH}(x, x').
\end{align}

Choose an arbitrary~$\varepsilon > 0$ and a bijection~$f\:[\Delta_1]\to [\R]$ satisfying~(\ref{LipIneq2}).

\textbf{Case~1.} Suppose that $f(\Delta_1) = Y \neq\mathbb{R}$. In other words, $d_{GH}\bigl(Y, \R\bigr) > 0$. 

Since $d_{GH}\bigl(Y, \R\bigr) > 0$, for abitrary $\lambda \neq 1, \lambda > 0$, we have $d_{GH}\bigl(\lambda Y, \R\bigr) = \lambda d_{GH}\bigl(Y, \R\bigr) \neq d_{GH}\bigl(Y, \R\bigr)$ and, hence, $\lambda Y\neq Y$. We put $X = f^{-1}(2Y)$. From~\ref{LipIneq2} and the inequality $d_{GH}\bigl(Y, \R\bigr) > 0$ it follows that $d_{GH}(X, \Delta_1) > 0$. We put $2Y_2 = f(\frac{3}{2}X),\,2Y_1 = f(\frac{1}{2}X)$. 

Let $d_{GH}(X, \Delta_1) = 2\rho$. Then from Theorem~\ref{thm: boundedcloud} we conclude that $d_{GH}(X,\frac{3}{2}X) = d_{GH}(X, \frac{1}{2}X) = \rho$. 

Inequalities~(\ref{LipIneq2}) imply that  $d_{GH}(2Y,\,2Y_i)\le (1+\varepsilon)\rho,\,i=1,\,2$, and $d_{GH}(2Y_1,\,2Y_2)\ge (1-\varepsilon)2\rho$. 

Dividing these by~$2$, we obtain $d_{GH}(Y,Y_i) \le \frac{1+\varepsilon}{2}\rho,\,i= 1,\,2$, and $d_{GH}(Y_1,Y_2)\ge (1-\varepsilon)\rho$.

We put $X_i = f^{-1}(Y_i),\,i = 1,\,2$. By once again applying Inequalities~(\ref{LipIneq2}), we obtain that $d_{GH}(\Delta_1, X_i)\le \frac{(1+\varepsilon)^2}{2}\rho,\,i=1,\,2$, and $d_{GH}(X_1,X_2) \ge (1-\varepsilon)^2\rho$. By Theorem~\ref{thm: boundedcloud}, $$(1-\varepsilon)^2\rho\le d_{GH}(X_1, X_2)\le \frac{1}{2}\max\Bigl\{d_{GH}(X_1, \Delta_1), d_{GH}(X_2, \Delta_1)\Bigr\}\le \frac{(1+\varepsilon)^2\rho}{4}. $$ Therefore, $$3\varepsilon^2-10\varepsilon + 3\le 0,$$
and, hence, $$\varepsilon \ge \frac{5 - \sqrt{5^2-3\cdot 3}}{3} = \frac{1}{3} > 0.$$

\textbf{Case~2.} Suppose that~$f(\Delta_1) = \mathbb{R}$. According to Theorem~\ref{theorem: geodesicthroughR}, we can choose metric spaces $X$ and~$Y$ such that $r := d_{GH}(X, \R) = d_{GH}(Y, \R) > 0$ and $d_{GH}(X, Y) = 2r$.

We put $X_1 = f^{-1}(\mathbb{Z}),\,X_2 = f^{-1}(\tilde{\mathbb{R}})$. By Inequalities~\ref{LipIneq2}, 
$d_{GH}(X_i,\Delta_1)\le (1+\varepsilon)r,\,i = 1,\,2$, and $d_{GH}(X_1,X_2)\ge (1-\varepsilon)\cdot 2r$. By Theorem~\ref{thm: boundedcloud}, $$(1-\varepsilon)\cdot 2r \le d_{GH}(X_1, X_2)\le\frac{1}{2}\max\Bigl\{d_{GH}(X_1, \Delta_1), d_{GH}(X_2, \Delta_1)\Bigr\}\leqslant \frac{1}{2}(1+\varepsilon)r.$$
Therefore, 
$$ 4 - 4\varepsilon \le 1 + \varepsilon \Longleftrightarrow \varepsilon \ge \frac{3}{5}.$$

Summing up, in both cases we have shown that  $\varepsilon\ge\frac{1}{3}$ which implies that the Lipschitz distance between $[\Delta_1]$ and~$[\mathbb{R}]$ cannot equal zero.
\end{proof}

Note that in fact we obtain a stronger inequality. Namely, for an arbitrary bijection $f\colon X\to Y$, if $\ln\max\bigl\{\dil f,\,\dil f^{-1}\bigr\}\le \delta$, according to the proof of Lemma~\ref{LipDefinition}, the following inequality holds: $\min\bigl\{e^\delta-1, 1-e^{-\delta}\bigr\} \geqslant \min\bigl\{\frac{1}{3}, \frac{3}{5}\bigr\}$. It follows that $\delta \ge \min\bigl\{\ln\frac{4}{3}, \ln\frac{3}{2}\bigr\} = \ln\frac{4}{3}$. Thus, 

\begin{corollary}
The inequality holds $d_L\bigl([\Delta_1],\,[\mathbb{R}]\bigr)\ge \ln\frac{4}{3}$.
\end{corollary}


\end{document}